\documentclass[11pt,leqno]{amsart}
\usepackage[top=1.5in, bottom=0.9in, left=0.9in, right=0.9in]{geometry}
\usepackage{amscd}
\usepackage{color}
\usepackage[symbol]{footmisc}
\usepackage{amssymb}
\usepackage{amsfonts}
\usepackage{latexsym}
\usepackage{verbatim}
\usepackage{bbm}
\usepackage[shortlabels]{enumitem}
\usepackage[backref=page]{hyperref}
\renewcommand*{\backref}[1]{}
\renewcommand*{\backrefalt}[4]{%
	\ifcase #1 (Not cited).%
	\or        (Cited on page~#2).%
	\else      (Cited on pages~#2).%
	\fi}
\hypersetup{colorlinks=true,linkcolor=blue,citecolor=blue,urlcolor=black}

\newcommand{\R}{\mathbb{R}}
\newcommand{\C}{\mathbb{C}}
\renewcommand{\H}{\mathbb{H}}
\newcommand{\Ker}{\mathrm{Ker}}
\renewcommand{\epsilon}{\varepsilon}
\renewcommand{\phi}{\varphi}

\theoremstyle{plain}
\newtheorem{thm}{Theorem}[section]
\newtheorem{prop}[thm]{Proposition}
\newtheorem{lem}[thm]{Lemma}
\newtheorem{cor}[thm]{Corollary}

\theoremstyle{definition}

\newtheorem{rmk}[thm]{Remark}
\newtheorem{ex}[thm]{Example}
\title{On balanced HKT manifolds}

\begin{document}
 

\author{Giovanni Gentili}
\address{(Giovanni Gentili)  Dipartimento di Matematica G. Peano, Universit\`a di Torino, Via Carlo Alberto 10, 10123, Torino, Italy.}
\email{giovanni.gentili@unito.it}

\author{Mehdi Lejmi}
\address{(Mehdi Lejmi) Department of Mathematics, Bronx Community College of CUNY, Bronx, NY 10453, USA.}
\email{mehdi.lejmi@bcc.cuny.edu}


\begin{abstract}
We prove the openness of the balanced HKT cone within the cone of HKT structures on a compact hypercomplex manifold $(M,I,J,K)$. We also study the Lie algebra of hyperholomorphic vector fields of type $(1,0)$ with respect to $I$, with particular emphasis on the case when there exists a compatible balanced HKT metric. These fields exhibit a strict interplay with the balanced HKT structure, for instance, we prove a harmonicity property for $(1,0)$-forms dual to hyperholomorphic vector fields. We also show non-existence of hyperholomorphic $(1,0)$-vector fields on some hypercomplex manifolds admitting a HKT--Einstein metric.
\end{abstract}

\maketitle

\section{Introduction}

Hyperk\"ahler with torsion (HKT) structures have been introduced in the physical literature by Howe and Papadopoulos \cite{HP}. They play a significant role in certain supersymmetric sigma models and other physical theories. These structures belong to the realm of hyperhermitian geometry and represent an interesting weakening of the hyperk\"ahler condition. We refer to Section \ref{Sec:pre} for essential definitions.

The attention devoted to HKT and hypercomplex structures has recently increased: among others, we mention the very recent work \cite{AB1,AB2,AGT,AT,BFG,BFGV,FG,GTar,LT,Pap1,Pap2,Witten}. Part of this renewed interest of the mathematical community is due to the influential quaternionic Calabi conjecture formulated by Alesker and Verbitsky \cite{AV} in 2010. On a HKT manifold $(M,I,J,K,g,\Omega)$, the form $\Omega\in \Lambda^{2,0}_I(M)$ induces a nowhere vanishing section of the canonical bundle $\Omega^n\in \Lambda^{2n,0}_I(M)$, where $2n=\dim_\C(M)$. The expectation is that on a compact HKT manifold, any q-positive holomorphic section of the canonical bundle arises in this way, namely as the top wedge power of another HKT metric. The new HKT metric, which would be the ``quaternionic version'' of a Calabi--Yau metric in K\"ahler geometry, is then also balanced and (first) Chern--Ricci flat. The conjecture has only been fully proved with the additional assumption of having a hyperk\"ahler metric by Dinew and Sroka \cite{DS} (see also \cite{BGV}), which unfortunately excludes from the picture all the interesting geometric scenarios in the non-K\"ahler framework.

\medskip

Our first main result, motivated by the conjecture of Alesker and Verbitsky, concerns the cone of balanced HKT cohomology classes. A HKT metric $\Omega$ on a hypercomplex manifold $(M,I,J,K)$ determines a quaternionic Bott--Chern cohomology class in $H^{2,0}_{\mathrm{BC}}(M)$. All such classes admitting a HKT metric as a representative form an open cone inside $H^{2,0}_{\mathrm{BC}}(M)$. Also, we can consider the smaller cone of cohomology classes that admit a balanced HKT representative. Should the conjecture of Alesker and Verbitsky turn out to be true, it would imply that on compact HKT manifolds with holomorphically trivial canonical bundle the balanced HKT cone coincides with the HKT cone. We prove, analogously to the classical result of LeBrun and Simanca \cite{LBS} regarding extremal K\"ahler metrics, that a given balanced HKT class has an open neighborhood in $H^{2,0}_{\mathrm{BC}}(M)$ such that each element of this neighborhood is represented by a balanced HKT metric:

\begin{thm}\label{Thm:main1}
Let $(M,I,J,K)$ be a compact hypercomplex manifold. Then the balanced HKT cone is open in the HKT cone.
\end{thm}

It should be remarked that a notion of extremal HKT metric is still missing in the literature. However, we could take into account HKT metrics of constant Chern--scalar curvature as well as HKT--Einstein metrics as the analogue of constant scalar curvature K\"ahler metrics and K\"ahler--Einstein metrics, respectively. From this point of view, balanced HKT metrics are those HKT--Einstein metrics with vanishing Chern--scalar curvature, just as Calabi--Yau metrics are K\"ahler--Einstein metrics with vanishing Riemannian scalar curvature. Surprisingly, unlike what happens in the K\"ahler framework, compact hypercomplex manifolds cannot support negative constant Chern--scalar curvature HKT metrics (see \cite{FG}). 

\medskip

The rest of the paper deals with the Lie algebra of hyperholomorphic vector fields on a hypercomplex manifold $(M,I,J,K)$, i.e. vector fields such that
\[
\mathcal{L}_XI=\mathcal{L}_XJ=\mathcal{L}_XK=0.
\]
Our interest towards these structures has several motivations. First, hyperholomorphic vector fields form the Lie algebra of the automorphism group of a hypercomplex manifold and thus they deserve to be studied on their own as they convey deep and meaningful information of the hypercomplex structure. Second, the presence of such fields gives constraints to the possible metric geometries arising on a given manifold, specially for the case of balanced HKT structures, which is our primary interest also in view of the quaternionic Calabi conjecture. Third, this kind of fields have been used as a tool in some constructions of new examples, most notably in reduction techniques, such as the hyperk\"ahler reduction \cite{HKLR}, the hypercomplex reduction of Joyce \cite{Joycered} and the HKT reduction of Grantcharov, Papadopoulos, and Poon \cite{GPP} (see also \cite{PPS}). Related investigations have been carried out, e.g., in \cite{OPS,Pap1,Pap2,Ver03}. For starters, we obtain a Hodge--type decomposition for hyperholomorphic vector fields in $T^{1,0}_I(M)$, which allows us to prove the following:

\begin{thm}\label{Thm:main3}
Let $(M,I,J,K,g,\Omega)$ be a compact balanced HKT manifold. Suppose that $X\in T^{1,0}_I(M)$ is a hyperholomorphic vector field and denote by $\alpha\in \Lambda^{1,0}_I(M)$ the form given by $\alpha=\Omega(X)=J\left(X\right)^{\flat}$. Then $\alpha$ and $J\bar \alpha$ are $\Delta_\partial$-harmonic. In particular, the Lie algebra of hyperholomorphic $(1,0)$-vector fields is given by Killing vector fields of constant length.
\end{thm}



The organization of the paper is the following. In Section \ref{Sec:pre}, we cover essential notions that will be useful throughout the paper. Section \ref{Sec:Open} is devoted to the proof of Theorem \ref{Thm:main1}. In Section \ref{Sec:harmonic}, we collect a few observations on harmonicity of forms. The rest of the paper deals with hyperholomorphic vector fields. More precisely, in Section \ref{Sec:1,0}, we investigate the properties of hyperholomorphic vector fields of type $(1,0)$ with respect to one of the complex structures. When there exists a compatible balanced HKT metric, the Lie algebra of such vector fields is abelian and consists of Killing vector fields of constant norm. Section \ref{Sec:abelian} is then focused on Lie algebras with a left-invariant abelian hypercomplex structure
. Finally, Section \ref{Sec:Einstein} deals with the interplay between the Einstein condition with a non-zero Einstein constant and hyperholomorphic $(1,0)$-vector fields on compact hypercomplex manifolds, showing that the latter cannot exist.

\medskip 

\noindent {\bf Acknowledgment.} The present work was initiated when the first author was visiting Mehdi Lejmi at the CUNY Graduate Center of New York, he would like to thank him deeply for the delightful hospitality and financial support. He is also grateful to INdAM as a recipient of a scholarship aimed at carrying out a research project abroad. The first author is indebted to Beatrice Brienza and Luigi Vezzoni for useful discussions and for showing interest in the paper. Finally, he acknowledges Elia Fusi for a careful reading of an early draft of the paper. The second author acknowledges support provided by the Simons Foundation Grant \#636075.

\section{Preliminaries}\label{Sec:pre}
In this section we review some basic facts concerning hypercomplex and HKT geometry with the purpose of fixing terminologies and notations adopted throughout the paper. Let $(M,I,J,K,g,\Omega)$ be a hyperhermitian manifold. Here $(I,J,K)$ is a hypercomplex structure, $g$ is a Riemannian metric compatible with $(I,J,K)$ and
\begin{equation}\label{Omega}
\Omega=\frac{g(J\cdot,\cdot)+\sqrt{-1}\,g(K\cdot,\cdot)}{2}
\end{equation}
is a non-degenerate $(2,0)$-form with respect to $I$ that is q-real and q-positive:
\[
J\Omega=\bar \Omega, \qquad \Omega(X,JX)>0,\, \text{ for all non-zero } X\in TM.
\]
As is customary, we will occasionally abuse language and call $\Omega$ a ``hyperhermitian metric'', this is motivated by the fact that a $(2,0)$-form with the properties above completely determines a hyperhermitian metric related to $\Omega$ as in \eqref{Omega}. A hyperhermitian structure is called \textbf{HKT} if
\[
\partial \Omega=0,
\]
where $\partial$ is the (conjugate) Dolbeault operator taken with respect to $I$. The \textbf{Lee form} of an Hermitian manifold $(M,J,g)$ is defined as the $1$-form $\theta:=J\delta F$, where $F$ is the fundamental form associated to $(J,g)$ and $\delta$ the codifferential. Our convention for the action of $J$ on a $k$-form $\alpha$ is 
\[
J\alpha=\alpha(J^{-1}\cdot,\dots,J^{-1}\cdot)=(-1)^k\alpha(J\cdot,\dots,J\cdot)\,.
\]
When $ \theta=0 $, equivalently $dF^{m-1}=0$, $m=\dim_\C(M)$, the Hermitian structure $(J,g)$ is called \textbf{balanced}. On a hyperhermitian manifold $(M,I,J,K,g,\Omega)$ the Lee forms of all Hermitian structures $(g,I),(g,J),(g,K)$ coincide (see \cite[Proposition 3.4]{FG}) and we can talk about the Lee form $\theta$ of $(M,I,J,K,g)$ without ambiguity. Furthermore as follows from the computations in \cite{Ver09} (see also \cite[Lemma 2.2]{BGV}), on a HKT manifold the Lee form satisfies
\begin{equation}\label{eq:Lee}
\partial \bar \Omega^n=\theta^{1,0} \wedge \bar \Omega^n,
\end{equation}
where $(\cdot)^{p,q}$, here and in the following, denotes the $(p,q)$-part of a form taken with respect to $I$. Thus, the structure is balanced HKT if and only if the top wedge power of the hyperhermitian metric is a holomorphic $(2n,0)$-form, where $2n=\dim_\C(M)$. From \eqref{eq:Lee} we also deduce the following formula
\begin{equation}\label{eq:del*Omega}
\partial^* \Omega= -*\partial * \Omega =-*\partial\left( \frac{\Omega^{n-1}\wedge \bar \Omega^n}{(n-1)!n!} \right) =-*\left( \theta^{1,0} \wedge \frac{\Omega^{n-1}\wedge \bar \Omega^n}{(n-1)!n!} \right) =-J\theta^{0,1}
\end{equation}
where we used that for every $\alpha \in \Lambda^{1,0}_I(M)$ the Hodge-star operator acts as
\[
*\alpha=J\bar \alpha \wedge \frac{\Omega^{n-1}\wedge \bar \Omega^n}{(n-1)!n!}.
\]

Another characterization of the balanced condition for HKT manifolds is the vanishing of the Chern--Ricci form of all Hermitian structures $(g,I),(g,J),(g,K)$ (see, e.g. \cite[Lemma 2.2]{BGV}). The \textbf{Chern--Ricci form} $\rho_I$ of a Hermitian structure $(g,I)$ is given locally by
\[
\rho_I=-i\partial \bar \partial \log \det(g).
\]
The trace of the Chern--Ricci form with respect to $F$ yields the \textbf{Chern--scalar curvature}. Again, it turns out that all Chern--scalar curvatures on a hyperhermitian manifold agree (see \cite[Proposition 3.10]{FG}) and there is no ambiguity when referring to the Chern--scalar curvature $s^{\mathrm{Ch}}(\Omega)$ of a hyperhermitian structure $(I,J,K,g,\Omega)$. We also report the following explicit formula
\begin{equation}\label{eq:sCh}
s^{\mathrm{Ch}}(\Omega)=2\Lambda  \left(\partial_J (\theta^{1,0}) \right):=2n\frac{\partial_J( \theta^{1,0}) \wedge \Omega^{n-1}}{\Omega^n},
\end{equation}
where $\partial_J=J\bar \partial J^{-1} \colon \Lambda^{p,q}_I(M) \to \Lambda^{p+1,q}_I(M)$. Here, $\Lambda$ is the dual of the Lefschetz operator $\Omega \wedge -$. It will also be useful to recall the following identities proved in \cite{Ver02}:
\begin{equation}\label{HKTidentities}
    [\Lambda,\partial]=-\partial_J^*,\qquad [\partial_J,\Lambda]=-\partial^*.
\end{equation}

As advocated in \cite{FG} balanced HKT structures should be seen as a subclass of HKT structures satisfying an appropriate Einstein condition. A HKT manifold $(M,I,J,K,g,\Omega)$ is called \textbf{HKT--Einstein} if
\[
\partial_J(\theta^{1,0})=\lambda \Omega,
\]
for a real constant $\lambda$. Equivalently
\[
\frac{\rho_L-L'\rho_L}{2}=\lambda F_L,
\]
where $L,L'\in \{I,J,K\}$ anticommute and $F_L$ is the fundamental form of $(L,g)$. From \eqref{eq:sCh} it is clear that a HKT--Einstein metric has constant Chern--scalar curvature equal to $2n\lambda$. Moreover, as anticipated, balanced HKT metrics on compact hypercomplex manifolds are precisely those HKT--Einstein metrics with vanishing Chern--scalar curvature.

\medskip

On a hypercomplex manifold $(M,I,J,K)$ there is a unique torsion-free connection $\nabla$ that preserves the hypercomplex structure $(I,J,K)$. This connection is named after Obata, who discovered it \cite{Obata}. By the holonomy principle, the holonomy group of $\nabla$ lies inside $\mathrm{GL}(n,\H)$, where $2n=\dim_{\C}(M)$ and $\H$ is the algebra of quaternions. Whenever $\mathrm{Hol}(\nabla)\subseteq \mathrm{SL}(n,\H):=[\mathrm{GL}(n,\H),\mathrm{GL}(n,\H)]$ the hypercomplex manifold $(M,I,J,K)$ is called a \textbf{$\mathrm{SL}(n,\H)$-manifold}. In addition, a hypercomplex manifold is $\mathrm{SL}(n,\H)$ if and only if it admits an Obata-parallel $(2n,0)$-form $\Phi \in \Lambda^{2n,0}_I(M)$. Such a form can further be chosen q-positive, i.e. it is the $n$\textsuperscript{th} wedge power of a hyperhermitian metric $\Omega$. As a consequence, the canonical bundle has to be holomorphically trivial. In \cite{Ver07} Verbitsky proved in the compact setting that when there exists a compatible HKT metric the $\mathrm{SL}(n,\H)$-condition is actually equivalent to having holomorphically trivial canonical bundle; in particular, we see from \eqref{eq:Lee} that the existence of a balanced HKT structure forces the manifold to be $\mathrm{SL}(n,\H)$. Furthermore, for compact manifolds, the class of those admitting a balanced HKT structure is precisely the intersection of the class of those carrying HKT--Einstein metrics and that of $\mathrm{SL}(n,\H)$-manifolds. The HKT assumption on Verbitsky's result cannot be removed, indeed as shown by Andrada and Tolcahier \cite[Example 6.3]{AT} there might be compact hypercomplex manifolds that are not $\mathrm{SL}(n,\H)$ but have a holomorphically trivial canonical bundle. A full characterization of the $\mathrm{SL}(n,\H)$-condition on compact hypercomplex manifolds was recently given in \cite[Theorem 1.2]{FG} in terms of the holomorphic triviality of the canonical bundle and an additional metric condition that is weaker than HKT. Also, on compact HKT manifolds, the $\mathrm{SL}(n,\H)$-condition is equivalent to the validity of the $\partial\partial_J$-lemma: every $\partial$-closed, $\partial_J$-exact $(p,0)$-form (with respect to $I$) is $\partial \partial_J$-exact (see \cite[Theorem 6]{GLV} and \cite[Corollary 1.3]{FG}).

\medskip

Given a hyperhermitian structure $(g,\Omega)$ on a $\mathrm{SL}(n,\H)$-manifold $(M,I,J,K,\Phi)$ we use two types of Hodge-star operator. The usual one $*\colon \Lambda^{p,q}_I(M) \to \Lambda^{2n-p,2n-q}_I(M)$, determined uniquely by the Riemannian structure:
\[
\alpha \wedge * \beta=g(\alpha,\beta)\frac{\Omega^n\wedge \bar \Omega^n}{(n!)^2}\,, \quad \text{for every }\alpha,\beta \in \Lambda^{p,q}_I(M),
\]
and the operator $\star_\Phi \colon \Lambda^{p,0}_I(M) \to \Lambda^{2n-p,0}_I(M)$ introduced in \cite[Section 6]{LW} with the aid of the holomorphic form $\Phi$:
\[
\alpha \wedge \star_\Phi \beta\wedge \bar \Phi=g(\alpha,\beta)\frac{\Omega^n}{n!}\wedge \bar \Phi\,, \quad \text{for every }\alpha,\beta \in \Lambda^{p,q}_I(M).
\]
These allow one to define, as usual, the adjoint operators $\partial^*=-*\partial\,*$, $\partial_J^*=-*\partial_J\,*$, $\partial^{\star_\Phi}=-\star_\Phi\partial\,\star_\Phi$, $\partial_J^{\star_\Phi}=-\star_\Phi\partial_J\,\star_\Phi$. When taken with respect to a balanced HKT metric, it has been observed in \cite[Section 3.2]{GTar} that the adjoints with respect to $*$ and $\star_\Phi$ coincide. Furthermore, we recall that, on compact balanced HKT manifolds, the following Laplacians coincide \cite[Proposition 3.4]{GTar}:
\[
\Delta_\partial:=\partial \partial^*+\partial^*\partial, \qquad \Delta_{\partial_J}=\partial_J\partial_J^*+\partial_J^*\partial_J.
\]

\section{Openness of the balanced HKT cone}\label{Sec:Open}

Let $(M,I,J,K)$ be a hypercomplex manifold and fix a hyperhermitian structure $(g_0,\Omega_0)$ on it. If the form $\Omega_0$ is HKT it defines a class in the $(2,0)$ quaternionic Bott-Chern cohomology:
\[
H^{2,0}_{\mathrm{BC}}(M)=\frac{\{ \alpha \in \Lambda^{2,0}_IM \mid \partial \alpha =\partial_J \alpha=0\}}{\partial \partial_J (C^\infty(M,\C))}\,.
\]
If the manifold $M$ is compact then $H^{2,0}_{\mathrm{BC}}(M)$ is finite-dimensional, as it is isomorphic to the space of $(2,0)$ quaternionic Bott-Chern harmonic forms (see \cite{GLV}):
\[
\mathcal{H}^{2,0}_{\mathrm{BC}}(M,\Omega_0)=\{ \alpha \in \Lambda^{2,0}_IM \mid \partial \alpha=\partial_J \alpha = \partial_J^*\partial^* \alpha=0\}\,.
\]
Indeed, $\mathcal{H}^{2,0}_{\mathrm{BC}}(M,\Omega_0)$ coincides with the kernel of the fourth order elliptic operator
\[
\Delta_{\mathrm{BC}}=\partial^*\partial + \partial_J^*\partial_J+ \partial \partial_J \partial_J^* \partial^*+\partial_J^*\partial^*\partial \partial_J+\partial_J^*\partial \partial^*\partial_J+\partial^*\partial_J \partial_J^* \partial
\]
and we have a decomposition
\begin{equation}\label{eq:BCdec}
\Ker\left(\partial\vert_{\Lambda^{2,0}_I(M)}\right)\cap \Ker\left(\partial_J\vert_{\Lambda^{2,0}_I(M)}\right)=\mathcal{H}^{2,0}_{\mathrm{BC}}(M,\Omega_0) \oplus \partial \partial_J(C^\infty(M,\C))\,.
\end{equation}
Since any HKT form $\Omega$ on $(M,I,J,K)$ determines a class $[\Omega]_{\mathrm{BC}}\in H^{2,0}_{\mathrm{BC}}(M)$, it makes sense to consider the set of all $(2,0)$ quaternionic Bott-Chern cohomology classes admitting a representative which is a HKT form:
\[
\mathcal{K}:=\{ \Theta \in H^{2,0}_{\mathrm{BC}}(M)\mid \text{there exists } \alpha \in \Theta \text{ such that }\alpha>0 \}\,.
\]
The set $\mathcal{K}$ is a convex open cone inside the vector space $H^{2,0}_{\mathrm{BC}}(M)$ and we shall call it the \textbf{HKT cone} of the hypercomplex manifold $(M,I,J,K)$. We are interested in studying the subcone of cohomology classes containing a balanced HKT representative. More precisely, given a balanced HKT class we aim to find a balanced HKT representative for all classes sufficiently close it.

The decomposition \eqref{eq:BCdec} restricts to q-real forms as follows:
\begin{equation}
\Ker\left(\partial\vert_{\Lambda^{2,0}_{I,\R}(M)}\right)\cap \Ker\left(\partial_J\vert_{\Lambda^{2,0}_{I,\R}(M)}\right)=\mathcal{H}^{2,0}_{\R}(M,\Omega_0) \oplus \partial \partial_J(C^\infty(M,\R)),
\end{equation}
where $\mathcal{H}^{2,0}_\R(M,\Omega_0)$ denotes the space of quaternionic Bott-Chern harmonic q-real forms with respect to $\Omega_0$ of type $(2,0)$. In particular, any HKT metric $ \Omega$ can be written uniquely as 
\[
\Omega=\alpha+\partial \partial_J f,
\]
with $(\alpha,f)\in \mathcal{H}^{2,0}_{\R}(M,\Omega_0)\times C^\infty_0(M,\R)$, where $\alpha$ is the harmonic representative of $[\Omega]_{\mathrm{BC}} $ with respect to $\Omega_0$ and $C^\infty_0(M,\R)$ is the space of real-valued smooth functions with zero mean:
\[
C^\infty_0(M,\R):=\left\{ f\in C^\infty(M,\R) \, \Big \vert  \int_M f\, \Omega_0^{n}\wedge \bar \Omega_0^n=0 \right\}.
\]

Let $\mathcal{M}_{\mathrm{HKT}}$ denote the space of all HKT forms on $M$ compatible with the fixed hypercomplex structure $(I,J,K)$ and define the map
\[
\Psi \colon \mathcal{M}_{\mathrm{HKT}} \to \mathcal{H}^{2,0}_\R(M,\Omega_0)\times C^\infty_0(M,\R)\,, \qquad \Omega \mapsto \left(\alpha,\hat s^{\mathrm{Ch}}(\Omega)\right),
\]
where $\hat s^{\mathrm{Ch}}(\Omega)$ is the projection of the Chern scalar curvature onto $C^\infty_0(M,\R)$:
\[
\hat s^{\mathrm{Ch}}:=s^{\mathrm{Ch}}(\Omega)- \frac{1}{\mathrm{Vol}(M,\Omega_0)} \int_M s^{\mathrm{Ch}}(\Omega)\, \frac{\Omega_0^{n}\wedge \bar \Omega_0^n}{(n!)^2}.
\]
Recall that a metric $\Omega\in \mathcal{M}_{\mathrm{HKT}}$ is balanced if and only if its Chern-scalar curvature vanishes (see \cite[Lemma 3.14]{FG}), also, a balanced HKT metric is Bott-Chern harmonic thanks to \eqref{eq:del*Omega} and thus we have $\Psi(\Omega_0)=(\Omega_0,0)$. More generally, we make the following observation:

\begin{lem}\label{Lem:trivdef}
Let $(M,I,J,K,g_0,\Omega_0)$ be a compact balanced HKT manifold, then a $\partial$-closed, $\partial_J$-closed $(2,0)$-form is quaternionic Bott-Chern harmonic if and only if it has constant trace.
\end{lem}
\begin{proof}
Suppose $\alpha\in \Lambda^{2,0}_I(M) $ is $\partial$-closed and $\partial_J$-closed then, using the balanced HKT identities \eqref{HKTidentities} and the fact that the Chern Laplacian $\Delta_{\Omega_0}^{\mathrm{Ch}} \colon C^\infty(M,\R) \to C^\infty(M,\R)$ can be written as $\Delta_{\Omega_0}f=\Lambda_{\Omega_0}(\partial \partial_J f)$ we get
\begin{eqnarray*}
\Delta_{\Omega_0}^{\mathrm{Ch}} \left( \Lambda_{\Omega_0} \alpha\right)&=& \Lambda_{\Omega_0} (\partial \partial_J \Lambda_{\Omega_0} \alpha),\\
&=&-\partial_J^*\partial_J (\Lambda_{\Omega_0} \alpha),\\
&=&\partial_J^*\partial^*  \alpha
\end{eqnarray*}
and so we conclude thanks to the maximum principle.
\end{proof}

Our aim is to compute the derivative of the map $\Psi$ at $\Omega_0$ and then apply the inverse function theorem to show that all the Bott-Chern classes nearby $[\Omega_0]_{\mathrm{BC}}$ admit a q-positive representative with vanishing Chern-scalar curvature, which is thus a balanced HKT metric.

\begin{prop}\label{TPsi}
Let $(M,I,J,K,g_0,\Omega_0)$ be a compact balanced HKT manifold. The derivative of $\Psi$ at $\Omega_0$
\[
T_{\Omega_0}\Psi \colon T_{\Omega_0}\mathcal{M}_{\mathrm{HKT}} \cong \mathcal{H}^{2,0}_\R(M,\Omega_0)\oplus C^\infty_0(M,\R) \to \mathcal{H}^{2,0}_\R(M,\Omega_0)\oplus C^\infty_0(M,\R)
\]
is given by
\[
T_{\Omega_0}\Psi( \beta, h)=\left(\beta,-2\Delta_{g_0}^2 h\right)\,,
\]
where $\Delta_{g_0}$ is the de Rham laplacian of $g_0$.
\end{prop}
\begin{proof}
Let $\Omega_t$ be a curve in $\mathcal{M}_{\mathrm{HKT}}$ for $t\in (-\epsilon,\epsilon)$ such that $\Omega_t\vert_{t=0}=\Omega_0$. As before, we decompose $\Omega_t=\alpha_t+\partial \partial_J f_t $ where $(\alpha_t,f_t)\in \mathcal{H}^{2,0}_\R(M,\Omega_0)\oplus C^\infty_0(M,\R)$ is such that $ (\alpha_0,f_0) =(\Omega_0,0) $. For any quantity $A_{t}$ depending on $t$ in the following we shall use the notation $\dot A:=\frac{d}{dt}\vert_{t=0} A_{t}$. First, in order to avoid trivial deformations of $\alpha$ by rescalings, we take advantage of Lemma \ref{Lem:trivdef} and impose the condition
\begin{equation}\label{eq:trivdef}
\Lambda_{\Omega_0} \dot \alpha =0\,.
\end{equation}
Now, we begin by computing the variation of the volume form
\begin{eqnarray*}
\frac{d}{dt}\Big \vert_{t=0} \Omega_{t}^{n}&=&n (\dot \alpha +\partial \partial_J \dot f) \wedge \Omega^{n-1}_0,\\
&=&\Delta^{\mathrm{Ch}}_{\Omega_0} \dot f\, \Omega_0^{n},
\end{eqnarray*}
where we used \eqref{eq:trivdef}. Let $\theta_t$ be the Lee form of $\Omega_t$, then from \eqref{eq:Lee} we deduce
\begin{eqnarray*}
\dot \theta\wedge \bar \Omega_0^n+\Delta_{\Omega_0}^{\mathrm{Ch}}\dot f\, \theta_0\wedge \bar \Omega_0^n&=&\partial \left( \Delta_{\Omega_0}^{\mathrm{Ch}} \dot f\, \bar \Omega_0^n \right),\\
&=&\partial \Delta_{\Omega_0}^{\mathrm{Ch}} \dot f \wedge  \bar \Omega_0^n + \Delta_{\Omega_0}^{\mathrm{Ch}} \dot f\,\theta_0 \wedge  \bar \Omega_0^n
\end{eqnarray*}
and so we obtain
\[
\dot \theta= d \Delta_{\Omega_0}^{\mathrm{Ch}} \dot f\,.
\]
We can now compute the variation of the Chern scalar curvature $\dot s^{\mathrm{Ch}}:=\frac{d}{dt}\vert_{t=0} s^{\mathrm{Ch}}(\Omega_t)$. From \eqref{eq:sCh} we have
\[
s^{\mathrm{Ch}}(\Omega_t)\Omega_t^n=2n\, \partial_J (\theta^{1,0}_t)\wedge \Omega^{n-1}_t
\]
hence
\[
\dot s^{\mathrm{Ch}} \Omega_0^n+s^{\mathrm{Ch}}(\Omega_0) \Delta_{\Omega_0}^{\mathrm{Ch}} \dot f \, \Omega_0^n =2n\, \partial_J (\dot \theta^{1,0})\wedge \Omega_0^{n-1}+2n\,\partial_J (\theta_0^{1,0}) \wedge \frac{d}{dt}\Big \vert_{t=0} \Omega_t^{n-1} \,.
\]
Now, since $\Omega_0$ is balanced $\theta_0=0$ and $s^{\mathrm{Ch}}(\Omega_0)=0$ so we infer
\begin{eqnarray*}
\dot s^{\mathrm{Ch}}&=&2n\, \partial_J (\dot \theta^{1,0})\wedge \Omega_0^{n-1},\\
&=&2n\,\partial_J\partial \Delta_{\Omega_0}^{\mathrm{Ch}} \dot f \wedge \Omega_0^{n-1},\\
&=&-2\left(\Delta_{\Omega_0}^{\mathrm{Ch}}\right)^2 \dot f\,,
\end{eqnarray*}
in order to conclude we only need to observe that in the balanced case the Chern Laplacian acting on functions coincides with the opposite of the de Rham Laplacian, which follows from the well-known formula (see \cite[pp. 502-503]{G}):
\[
\Delta^{\mathrm{Ch}}_{\Omega_0} h=-\Delta_{g_0}h-g_0(dh,\theta_{0}), \qquad h\in C^\infty(M,\R). \qedhere
\]
\end{proof}

We are ready to prove Theorem \ref{Thm:main1}.

\begin{thm}
Let $(M,I,J,K)$ be a compact hypercomplex manifold. Then the balanced HKT cone is open in the HKT cone.
\end{thm}
\begin{proof}
Pick any balanced HKT structure $(g_0,\Omega_0)$ on $(M,I,J,K)$, which, in particular, must be a $\mathrm{SL}(n,\H)$-manifold. Consider the following extension of $\Psi$
\[
\Psi \colon \mathcal{M}_{\mathrm{HKT}}^{k+4} \to \mathcal{H}^{2,0}_\R(M,\Omega_0)\oplus W^k_0(M,\R)
\]
where $\mathcal{M}_{\mathrm{HKT}}^{k+4}$ denotes the Banach manifold of HKT metrics of bounded Sobolev norm of order $k+4$. For $k$ large enough $\mathcal{M}_{\mathrm{HKT}}^{k+4}$ is contained in the space of all HKT metrics of regularity $C^4$.

By Proposition \ref{TPsi} the kernel of $T_{\Omega_0}\Psi\colon \mathcal{H}^{2,0}_\R(M,\Omega_0)\oplus W^{k+4}_0(M,\R) \to \mathcal{H}^{2,0}_\R(M,\Omega_0)\oplus W^k_0(M,\R)$ consists of those pairs $(\beta,h)$ such that
\[
\left(\beta,-2\Delta_{g_0}^2h\right)=(0,0).
\]
Ellipticity of $\Delta_{g_0}^2$ implies that $h$ must be smooth, but then the maximum principle forces $h$ to be constant, because $M$ is compact and so $h=0$ by the normalization condition. Hence $T_{\Omega_0}\Psi$ is injective. Surjectivity also follows from the fact that $\Delta_{g_0}^2$ is an isomorphism between the spaces $W^{k+4}_0(M,\R)$ and $W^k_0(M,\R)$. 

By the inverse function theorem for Banach manifolds $\Psi$ is an isomorphism between an open neighborhood of $\Omega_0$ in $\mathcal{M}_{\mathrm{HKT}}^{k+4}$ onto an open neighborhood of $(\Omega_0,0)$ in $\mathcal{H}^{2,0}_{\R}(M,\Omega_0)\oplus W^k_0(M,\R)$. Therefore, there exists $\epsilon>0$ small enough such that for every $\beta \in \mathcal{H}^{2,0}_{\R}(M,\Omega_0) $ with $\|\beta- \Omega_0 \| \leq \epsilon $ the equation $\Psi(\Omega)=(\beta,0)$ has a solution in $\mathcal{M}_{\mathrm{HKT}}^{k+4}$. This means that the Chern-scalar curvature of $\Omega$ is constant. Since locally we can write 
\[
s^{\mathrm{Ch}}(\Omega)=-\Delta^{\mathrm{Ch}}_\Omega \log \left( \frac{\Omega^n\wedge \bar \Omega^n}{\Omega^n_0\wedge \bar \Omega^n_0} \right)
\]
standard elliptic theory reveals that $\Omega $ is smooth; we refer to \cite{LBS} for the details. Being a HKT metric with constant Chern-scalar curvature on a compact $\mathrm{SL}(n,\H)$-manifold, $\Omega$ is balanced.
\end{proof}

\section{Harmonic forms}\label{Sec:harmonic}

\begin{prop}\label{p0even}
Let $(M,I,J,K,g,\Omega)$ be a compact balanced HKT manifold. If $\alpha \in \Lambda^{p,0}_I(M)$ is $\Delta_\partial$-harmonic then so is $J\bar \alpha$.
\end{prop}
\begin{proof}
We have $\Delta_\partial=\Delta_{\partial_J}$, hence $\alpha$ is $\partial$, $\partial^*$, $\partial_J$ and $\partial_J^*$-closed. Clearly
\[
\partial J\bar \alpha=J^{-1}J\partial J\bar \alpha =-JJ\partial J^{-1}\bar \alpha=-J\bar \partial_J \bar \alpha=0,
\]
and using the identities \eqref{HKTidentities}, together with the fact $\Lambda J=J\bar \Lambda$ we also conclude
\begin{eqnarray*}
\partial^* J\bar \alpha&=&\Lambda (\partial_J J\bar \alpha) -\partial_J \Lambda (J\bar \alpha),\\
&=& \Lambda (J\bar \partial \bar \alpha)-J\bar \partial \bar \Lambda (\bar \alpha),\\
&=& -J\bar \Lambda(\bar \partial \bar \alpha)-J\bar \partial_J^* \bar \alpha,\\
&=&0,
\end{eqnarray*}
hence $J\bar \alpha$ is $\Delta_\partial$-harmonic.
\end{proof}

\begin{cor}
Let $(M,I,J,K,g)$ be a compact balanced HKT manifold. Then $\dim H^{2p-1,0}_\partial(M) \equiv 0 \mod 2$ for all $p=1,\dots,n$ and $H^{2p,0}_\partial(M)\neq 0$ for $p=0,\dots,n$.
\end{cor}
\begin{proof}
The first assertion follows from Proposition \ref{p0even} together with the fact that $\alpha$ and $J\bar \alpha$ are independent if the form $\alpha \in \Lambda^{2p-1,0}_I(M)$ is of odd degree. The second part is a consequence of the fact that the powers $\Omega^p$ are $\Delta_\partial$-harmonic for every $p=0,\dots,n$, indeed they are evidently $\partial$-closed and also
\[
\partial^* \Omega^p=-*\partial*\Omega^p=-p!*\partial \left(\frac{\Omega^{n-p}\wedge \bar \Omega^n}{(n-p)!n!}\right)=0. \qedhere
\]
\end{proof}

\begin{prop}\label{parallel-harmonic}
Let $(M,I,J,K,\Phi,g)$ be a compact HKT $\mathrm{SL}(n,\H)$-manifold and $\nabla$ the Obata connection on it. If $\alpha \in \Lambda^{1,0}_I(M)$ is $\nabla$-parallel it is $\Delta_{\partial,\Phi}$-harmonic, where $\Delta_{\partial,\Phi}:=\partial\partial^{\star_\Phi}+\partial^{\star_\Phi}\partial$.
\end{prop}
\begin{proof}
If $\alpha \in \Lambda^{1,0}_I(M)$ is parallel also $J \bar \alpha$ is because $\nabla$ preserves $J$. Furthermore, as $\nabla$ is torsion-free, this implies that $\alpha$ and $J \bar \alpha$ are closed. In particular $\partial \alpha =0$ and $\partial J \bar \alpha =0$. We then conclude
\begin{eqnarray*}
\partial^{\star_\Phi} \alpha&=&-\star_{\Phi} \partial \star_{\Phi} \alpha,\\
&=& -\star_{\Phi}\, \partial \left(J \bar \alpha \wedge \frac{\Omega^{n-1}}{(n-1)!} \right),\\
&=& 0,
\end{eqnarray*}
and so $\alpha$ is $\Delta_{\partial,\Phi}$-harmonic.
\end{proof}

Example \ref{ex:balancedHKT} has the purpose of showing that the converse of Proposition \ref{parallel-harmonic} does not hold as, even in the balanced HKT case, a $(1,0)$-form can be $\Delta_\partial$-harmonic without being closed.

\begin{ex}\label{ex:balancedHKT}
We consider an example of a compact hypercomplex nilmanifold that is balanced HKT. The structure equations of the associated Lie algebra are given by
$$\left\{
    \begin{array}{llll}
        de^1=de^2=de^3=de^4=de^5=0,  \\
        de^6=e^1\wedge e^2+e^3\wedge e^4,\\
        de^7=e^1\wedge e^3+e^4\wedge e^2,\\
        de^8=e^1\wedge e^4+e^2\wedge e^3.
    \end{array}
\right.
$$
The hypercomplex structure is given by
$$
Ie_1=-e_2,\quad Ie_3=e_4,\quad Ie_5=\frac{1}{2}e_6,\quad Ie_7=e_8, $$
$$
Je_1=-e_3,\quad J e_2=-e_4,\quad Je_5=\frac{1}{2}e_7,\quad Je_6=-e_8,
$$
and the balanced HKT metric $\Omega$ is given by
$$
\Omega=-\phi^1\wedge\phi^2+2\phi^3\wedge\phi^4,
$$
where
$$
\phi^1=e^1+\sqrt{-1}e^2,\quad \phi^2=e^3-\sqrt{-1}e^4,\quad \phi^3=e^5-2\sqrt{-1}e^6,\quad \phi^4=e^7-\sqrt{-1}e^8.
$$
It is easy to check that
$$
\partial\phi^i=\partial_J\phi^i=0,\mbox{ for any } i.
$$
The dimension of $\Delta_\partial$-harmonic $(1,0)$-forms with respect to $I$ is thus $4$, but, for instance, $\phi^3$ is not closed. 
\end{ex}

\section{Hyperholomorphic \texorpdfstring{$(1,0)$}--vector fields}\label{Sec:1,0}
Let $(M,I,J,K,g,\Omega)$ be a hyperhermitian manifold. In this section we investigate the condition for a vector field of type $(1,0)$ with respect to $I$ to be hyperholomorphic. A vector field $Z\in T^{1,0}_I(M)$ is called \textbf{hyperholomorphic} if $\mathcal{L}_ZI=\mathcal{L}_ZJ=\mathcal{L}_ZK=0$. We start with the following characterization:

\begin{prop}\label{parallel-obata}
    Let $(M,I,J,K)$ be a hypercomplex manifold. Consider a vector field $X\in TM$ and denote $Z=X^{1,0}\in T^{1,0}_I(M)$ its $(1,0)$-part with respect to $I$. Then the following are equivalent:
    \begin{enumerate}
    \item $Z$ is hyperholomorphic;
    \item $X$, $IX$, are (real) hyperholomorphic;
    \item $J\bar Z$ is hyperholomorphic;
    \item $Z$ is parallel with respect to the Obata connection;
    \item $X$ is parallel with respect to the Obata connection.
\end{enumerate}

\end{prop}
\begin{proof}
Suppose $Z$ is hyperholomorphic, we obtain that
$$
\mathcal{L}_XI=\mathcal{L}_XJ=\mathcal{L}_XK=\mathcal{L}_{IX}I=\mathcal{L}_{IX}J=\mathcal{L}_{IX}K=0.
$$
Using the above identities and integrability, we compute
\begin{eqnarray*}
\mathcal{L}_{J\overline{Z}}I&=&\mathcal{L}_{JX}I-\sqrt{-1}\mathcal{L}_{KX}I,\\
&=&\left(\mathcal{L}_{JX}J\right)K+J\left(\mathcal{L}_{JX}K\right)-\sqrt{-1}\left(\mathcal{L}_{KX}J\right)K-\sqrt{-1}J\left(\mathcal{L}_{KX}K\right),\\
&=&J\left(\mathcal{L}_{KIX}K\right)+\sqrt{-1}\left(\mathcal{L}_{JIX}J\right)K,\\
&=&JK\left(\mathcal{L}_{IX}K\right)+\sqrt{-1}J\left(\mathcal{L}_{IX}J\right)K,\\
&=&0.
\end{eqnarray*}
On the other hand, since $\mathcal{L}_XJ=\mathcal{L}_{IX}J=0$, we have that $\mathcal{L}_{J\overline{Z}}J=0.$ Therefore we have established the equivalence of the first three assertions.

As $\nabla$ preserves $I$, $J$ and $K$ it is clear that the last two assertions are equivalent. Furthermore, since $\nabla$ is torsion-free, for every $L\in \{I,J,K\}$ and $Y\in TM$ we have
\begin{eqnarray*}
(\mathcal{L}_XL)(Y)&=&[X,LY]-L[X,Y]\\
&=&\nabla_XLY-\nabla_{LY}X-L\nabla_XY+L\nabla_YX\\
&=&-\nabla_{LY}X+L\nabla_YX
\end{eqnarray*}
therefore, every $\nabla$-parallel vector field is hyperholomorphic. Conversely, if $X$ and $I X$ are hyperholomorphic we have
\[
\nabla_{KY}X=K\nabla_YX=-J\nabla_YIX=-\nabla_{JY}IX=-I\nabla_{JY}X=-\nabla_{KY}X,
\]
for every $Y\in TM$ and so $X$ has to be parallel.
\end{proof}

\begin{cor}
Let $(M,I,J,K)$ be a compact connected hypercomplex manifold with Euler characteristic $\chi\neq 0$. Then, there is no non-trivial hyperholomorphic $(1,0)$-vector field.
\end{cor}
\begin{proof}
This is a consequence of Theorem~\ref{parallel-obata} and the Poincaré-Hopf Theorem.
\end{proof}

On a compact HKT $\mathrm{SL}(n,\H)$-manifold, it turns out that the Riemannian dual of a hyperholomorphic vector field has a special Hodge decomposition that can be compared to the Hodge decomposition of a holomorphic vector field on a K\"ahler manifold.

\begin{thm}\label{SL(n,H)-hyperholomorphic}
Let $(M,I,J,K,\Phi,g,\Omega)$ be a compact HKT $\mathrm{SL}(n,\H)$-manifold. Suppose that $Z\in T^{1,0}(M)$ is hyperholomorphic and denote by $\alpha\in \Lambda^{1,0}_I(M)$ the form given by $\alpha=\Omega(Z)=J\left(Z\right)^{\flat}$, where $\flat$ denotes the dual by the metric $g$. Then, 
$$\alpha=\left(\alpha\right)_H+\partial u+\partial_Jv,$$
where $(\cdot)_H$ is the harmonic part with respect to the Laplacian $\Delta_{\partial,\Phi}=\partial\partial^{\star_\Phi}+\partial^{\star_\Phi}\partial$, and $u,v$ are complex-valued functions. 
\end{thm}
\begin{proof}
First, we compute
$$\mathcal{L}_{Z}\Omega=\left(\bar{\partial}\Omega\right)(Z)+\partial\left(\Omega(Z)\right)+\bar{\partial}\left(\Omega(Z)\right).
$$
Since $Z$ is $I$-holomorphic then $\mathcal{L}_{Z}\Omega\in \Lambda^{2,0}_I(M)$. Hence, we have that
$$
\mathcal{L}_{Z}\Omega=\partial\left(\Omega(Z)\right)=\partial\alpha.
$$
So, $\mathcal{L}_{Z}\Omega$ is $\partial$-exact. It is also $\partial_J$-closed because $\mathcal{L}_{Z}J=0$ and thus $\partial_J \mathcal{L}_Z=\mathcal{L}_Z \partial_J$. Since the manifold is $\mathrm{SL}(n,\H)$ the $\partial\partial_J$-Lemma holds and so $\mathcal{L}_{Z}\Omega=\partial\partial_Jv,$
for some function $v.$ 
We consider the Hodge decomposition of $\alpha$ with respect to $\Delta_{\partial,\Phi}$:
$$\alpha=\left(\alpha\right)_H+\partial u+\partial^{\star_\Phi}\phi,
$$
where $\phi\in \Lambda^{2,0}_I(M)$. We obtain that
$$
\partial\alpha=\partial\partial^{\star_\Phi}\phi= \partial\partial_Jv.
$$
On the other hand, we claim that $\partial_Jv=-\partial^{\star_\Phi}(v\Omega)$. Indeed
\begin{eqnarray*}
\partial^{\star_\Phi}(v\Omega)&=&-\star_\Phi\partial\star_\Phi\left(\bar{v}\Omega\right),\\
&=&-\star_\Phi\partial\left(v\frac{\Omega^{n-1}}{(n-1)!}\right),\\
&=&-\star_\Phi\left(\partial v\wedge\frac{\Omega^{n-1}}{(n-1)!}\right),\\
&=& -\partial_Jv.
\end{eqnarray*}
We deduce that
$$\partial\partial^{\star_\Phi}\phi=-\partial\partial^{\star_\Phi}(v\Omega).
$$
Since $M$ is compact, we get
$$\partial^{\star_\Phi}\phi=-\partial^{\star_\Phi}(v\Omega)=\partial_Jv.
$$ The Theorem follows.
\end{proof}

Using the decomposition of Theorem \ref{SL(n,H)-hyperholomorphic} we can show that on balanced HKT manifolds the dual of the hyperholomorphic vector field is $\Delta_\partial$-harmonic. To prove this fact, we begin with the following observation:

\begin{lem}\label{coclosed}
Let $(M,I,J,K,\Phi)$ be a compact $\mathrm{SL}(n,\H)$-manifold. Then there exists a compatible hyperhermitian structure $(g,\Omega)$ such that $\Omega^n=\Phi$. Furthermore, if $X\in TM$ is a real vector field which is holomorphic (with respect to either $I$, $J$ or $K$), then $\Lambda(\mathcal{L}_X\Omega)=0$ and $\delta X^{\flat}=0$.
\end{lem}
\begin{proof}
The first assertion is straightforward as $\Phi$ is a q-positive $(2n,0)$-form it must be the $n$\textsuperscript{th} wedge power of a hyperhermitian metric.
We assume that $X$ is holomorphic with respect to $I$, the argument is similar for $J$ and $K$. By holomorphicity $\mathcal{L}_X$ commutes with $\bar \partial$, thus, we have
\begin{eqnarray*}
    0&=&\mathcal{L}_X\bar{\partial}\Phi,\\
    &=&\bar{\partial}\mathcal{L}_{X}\Omega^n,\\
    &=&n\bar{\partial}\left(\mathcal{L}_X\Omega\wedge\Omega^{n-1}\right),\\
    &=&\bar{\partial}\left(\Lambda\left(\mathcal{L}_X\Omega \right)\Omega^{n}\right),\\
    &=&\bar{\partial}\Lambda\left(\mathcal{L}_X \Omega \right)\wedge\Omega^{n}.
\end{eqnarray*}
It follows by compactness that $\Lambda(\mathcal{L}_X \Omega)$ is constant, namely $\mathcal{L}_X\Omega^n=c\,\Omega^n$ for some constant $c$. Therefore, since $\Omega^n \wedge \bar \Omega^n$ is, up to a constant, the Riemannian volume form, we have
\[
c^2 \Omega^n \wedge \bar \Omega^n=\mathcal{L}_X(\Omega^n \wedge \bar \Omega^n)=d\left( \iota_X(\Omega^n \wedge \bar \Omega^n) \right).
\]
Integrating and using Stokes' Theorem we deduce that $c=0$, i.e. $\mathcal{L}_X\Omega^n=0$. The conclusion follows as $\mathcal{L}_X(\Omega^n\wedge \bar \Omega^n)=\delta \left(X\right)^{\flat}\Omega^n\wedge \bar \Omega^n $.
\end{proof}

\begin{thm}\label{dualhol}
Let $(M,I,J,K,g,\Omega)$ be a compact balanced HKT manifold. Suppose that $Z\in T^{1,0}_I(M)$ is a hyperholomorphic vector field and denote by $\alpha\in \Lambda^{1,0}_I(M)$ the form given by $\alpha=J\left(Z\right)^{\flat}$. Then $\alpha$ and $J\bar \alpha$ are $\Delta_\partial$-harmonic. In particular, the Lie algebra of hyperholomorphic $(1,0)$-vector fields is given by Killing vector fields of constant length.
\end{thm}
\begin{proof}
Since $Z$ is hyperholomorphic, we can apply Lemma \ref{coclosed} with $\Omega^n=\Phi$ and deduce $\partial^\ast \alpha=0$. It follows from Theorem~\ref{SL(n,H)-hyperholomorphic} that 
$$\alpha=\left(\alpha\right)_H+\partial_Jv,
$$
where $(\cdot)_H$ is the harmonic part with respect to the Laplacian $\Delta_\partial$ and $u,v$ are some functions. We observe that $\mathcal{L}_{Z}\Omega=\partial\alpha=\partial\partial_Jv$. From Lemma \ref{coclosed} we deduce that $\Lambda(\partial\partial_Jv)=0$ and since $M$ is compact, by the maximum principle we conclude that $v$ is constant and thus $\partial \alpha=0$. The harmonicity of $J\bar \alpha$ follows from Proposition \ref{p0even}. Also, as $\mathcal{L}_Z\Omega=\mathcal{L}_ZJ=0$, we deduce $\mathcal{L}_Zg=0$ and so $Z$ is a Killing vector field i.e. $X,IX$ are real Killing vector fields, where $X^{1,0}=Z$. Finally, let $Y\in T^{1,0}_I(M)$ be another hyperholomorphic vector field, since $ \alpha$ is $\Delta_\partial$-harmonic and $Y$ is Killing, we get $\partial^\ast\partial\left( \alpha(Y)\right)=\mathcal{L}_Y\partial^\ast \alpha=0$, thus $\partial\left(\alpha(Y)\right)=0$. We deduce that $\alpha(Y)$ is constant, so choosing $Y=J\bar{Z}$ we get that $|Z|_g^2$ is constant. 
\end{proof}

\begin{cor}\label{equivalences}
Let $(M,I,J,K,g,\Omega)$ be a compact balanced HKT manifold. Then the following are equivalent for $Z\in T^{1,0}_I(M)$:
\begin{enumerate}
    \item $Z$ is hyperholomorphic;
    \item $Z$ is Killing (i.e. $X,IX$ are real Killing vector fields, where $X^{1,0}=Z$);
    \item $Z$ is parallel with respect to the Bismut connection.
\end{enumerate}
\end{cor}
\begin{proof}
If $Z$ is hyperholomorphic then we have shown in Theorem~\ref{dualhol} that it is Killing. The converse is a consequence of \cite[Theorem 1.2 (i) (a)]{GI} because balanced HKT structures are Chern--Ricci flat. The equivalence with the fact that $Z$ is Bismut-parallel follows from~\cite{OPS}.
\end{proof}

In the next example we show that if we drop the HKT assumption Theorem~\ref{dualhol} does not hold anymore.

\begin{ex}\label{balanced-nonHKT}
Consider the nilpotent Lie algebra in \cite[Example 3]{LW} with structure equations
$$\left\{
    \begin{array}{llll}
        de^1=de^2=de^3=de^4=de^5=de^6=de^7=de^8=0,  \\
        de^9=e^1\wedge e^5,\\
        de^{10}=e^1\wedge e^6,\\
        de^{11}=e^1\wedge e^7,\\
        de^{12}=e^1\wedge e^8,
    \end{array}
\right.
$$
and hypercomplex structure
$$
Ie_1=e_2,\quad Ie_3=e_4,\quad Ie_5=e_6,\quad Ie_7=e_8, \quad Ie_9=e_{10}, \quad Ie_{11}=e_{12}, 
$$
$$
Je_1=e_3,\quad J e_2=-e_4,\quad Je_5=e_7,\quad Je_6=-e_8, \quad Je_9=e_{11},\quad Je_{10}=-e_{12}.
$$
Since the structure constants are rational the corresponding simply connected nilpotent Lie group admits lattices and so the hypercomplex structure descends to the corresponding compact nilmanifold $M$. Note that $M$ carries no HKT metrics because the hypercomplex structure is not abelian (see \cite[Theorem 4.6]{BDV}). On the other hand, $M$ admits a quaternionic balanced metric which is also balanced (cf. \cite[Example 9.1]{FG}).  The Lie algebra of left-invariant hyperholomorphic vector fields is generated by $
e_1, e_2, e_3, e_4, e_9, e_{10}, e_{11}, e_{12}. $
Furthermore one could see that the space of $\Delta_\partial$-harmonic $(1,0)$-forms is generated by $\phi^1,\phi^2,\phi^3,\phi^4$, where $ \phi^i=e^{2i-1}-\sqrt{-1}e^{2i}.$ Therefore $\phi^5=e^{9}-\sqrt{-1}e^{10}$ and $\phi^6=e^{11}-\sqrt{-1}e^{12}$ are not harmonic, even though $e_9,e_{10},e_{11},e_{12}$ are hyperholomorphic, so the HKT assumption in Theorem \ref{dualhol} cannot be relaxed. We also note that even though $e_1$ is hyperholomorphic, $e_1$ is not Killing with respect to any left-invariant metric $g$ because, for instance
\[
(\mathcal{L}_{e_1}g)(e_5,e_9)=-g([e_1,e_5],e_9)-g(e_5,[e_1,e_9])=g(e_9,e_9)\neq 0.
\]
\end{ex}

\begin{ex}\label{rmk}
The converse of Theorem~\ref{dualhol} is false as there can be $\Delta_\partial$-harmonic forms on compact balanced HKT manifolds that do not come from hyperholomorphic vector fields. An example is provided in Example~\ref{ex:balancedHKT}. Indeed, in the notations of Example~\ref{ex:balancedHKT} all $1$-forms are $\Delta_\partial$-harmonic and yet the only linearly independent hyperholomorphic vector fields are 
$e_5,e_6,e_7,e_8$. Thus, $e_1,e_2,e_3,e_4$ are not hyperholomorphic but the corresponding forms $\phi^1,\phi^2$ are $\Delta_\partial$-harmonic.
\end{ex}



\begin{rmk}
As shown in the next examples, we remark that if $X\in TM$ is a real hyperholomorphic vector field on a HKT manifold, then $X^{1,0}$ is not necessarily hyperholomorphic. Also, the balanced condition can not be dropped in Corollary~\ref{equivalences}. Moreover, if $X\in TM$ is a real hyperholomorphic vector field on a balanced HKT manifold, then $X$ is not necessarily a (real) Killing vector field.  
\end{rmk}

\begin{ex}\label{Hopf}
Consider the Hopf surface $M=\mathrm{SU}(2)\times S^1$, which is described by a global frame $\{ e_1,e_2,e_3,e_4\} $ and corresponding coframe $\{ e^1,e^2,e^3,e^4\} $ with structure equations
\[
\begin{cases}
de^1=-2e^2\wedge e^3,\\
de^2=2e^1\wedge e^3,\\
de^3=-2e^1\wedge e^2,\\
de^4=0.
\end{cases}
\]
We endow $M$ with the left-invariant hypercomplex structure $(I,J,K)$ such that
\[
Ie_1=e_2, \quad Ie_3=e_4, \qquad Je_1=e_3, \quad Je_2=-e_4.
\]
Consider the left-invariant hyperhermitian metric
\[
\Omega=\phi^1\wedge \phi^2,
\]
where
\[
\phi^1=e^1-\sqrt{-1}e^2, \quad \phi^2=e^3-\sqrt{-1}e^4.
\]
The metric $\Omega$ is HKT for dimensional reasons. On the other hand, it is well-known that the Hopf surface admits no balanced metric.
We note that $e_4$ is hyperholomorphic, but $e_3$ is not. Moreover, $(e^4)^{1,0}$ is not $\partial$-closed. We also note that $e_i$ is Killing with respect to $g$ for all $i=1,\dots,4$. 
\end{ex}

\begin{ex}
Consider the nilpotent Lie algebra $\mathfrak{g}= \langle e_1,\dots,e_{12} \rangle $ with structure equations
\begin{align*}
[e_1,e_5]=[e_2,e_6]=[e_3,e_7]=[e_4,e_8]&=e_9,\\
[e_1,e_6]=-[e_2,e_5]=-[e_3,e_8]=[e_4,e_7]&=e_{10},\\
[e_1,e_7]=[e_2,e_8]=-[e_3,e_5]=-[e_4,e_6]&=e_{11},\\
[e_1,e_8]=-[e_2,e_7]=[e_3,e_6]=-[e_4,e_5]&=e_{12},
\end{align*}
and abelian hypercomplex structure
\[
Ie_1=e_2, \quad Ie_3=e_4,\quad Ie_5=e_6, \quad Ie_7=e_8, \quad Ie_9=e_{10}, \quad Ie_{11}=e_{12},
\]
\[
Je_1=e_3, \quad Je_2=-e_4,\quad Je_5=e_7, \quad Je_6=-e_8, \quad Je_9=e_{11}, \quad Je_{10}=-e_{12}.
\]
Note that $e_1$ is hyperholomorphic, however $e_2=Ie_1$ is not $J$-holomorphic as
\[
[e_2,Je_5]-J[e_2,e_5]=[e_2,e_7]+Je_{10}=-e_{12}-e_{12}\neq 0.
\]
It follows that $e_1$ is hyperholomorphic but not Obata parallel, namely $(e_1)^{1,0}$ is not hyperholomorphic. Furthermore, the hyperhermitian metric
\[
g=\sum_{i=1}^{12}e^i\otimes e^i
\]
is balanced HKT but $e_1$ is not Killing because
\[
(\mathcal{L}_{e_1}g)(e_5,e_9)=-g([e_1,e_5],e_9)-g(e_5,[e_1,e_9])=-1\neq 0.
\]
\end{ex}


\section{Left-invariant abelian hypercomplex structures}\label{Sec:abelian}

In this section we focus on left-invariant abelian hypercomplex structures on Lie algebras.

\begin{thm}\label{abeliancenter}
Let $G$ be a Lie group endowed with a left-invariant abelian hypercomplex structure $(I,J,K)$. Then the Lie algebra of hyperholomorphic $(1,0)$-vector fields is given by the center of $TG\otimes\mathbb{C}.$
\end{thm}
\begin{proof}
The Obata connection is given by~\cite{Sol}:
$$
\nabla_XY=\frac{1}{2}\left([X,Y]+I[IX,Y]-J[X,JY]+K[IX,JY]    \right),
$$
for any pair of vector fields $X,Y$. Since the hypercomplex structure is abelian we obtain
\begin{eqnarray*}
\nabla_XY=\frac{1}{2}\left([X,Y]-I[X,IY]-J[X,JY]-K[X,KY]\right),
\end{eqnarray*}
so it is clear that if $X\in T^{1,0}_I(G)$ lies in the center it is Obata parallel and thus hyperholomorphic by Proposition~\ref{parallel-obata}. Conversely, if $X\in T^{1,0}_I(G)$ is a hyperholomorphic vector field, 
\begin{eqnarray*}
\nabla_XY&=&\frac{1}{2}\left([X,Y]-I[X,IY]-J[X,JY]-K[X,KY]\right),\\
&=&\frac{1}{2}\left([X,Y]+[X,Y]+[X,Y]+[X,Y]\right),\\
&=&2[X,Y].
\end{eqnarray*}
On the other hand, as $\nabla$ is torsion-free we also have
\begin{equation}
[X,Y]=\nabla_XY.
\end{equation}
We obtain that $[X,Y]=0$, for any vector field $Y.$ The theorem follows.
\end{proof}

\begin{ex}
We exhibit an example of a Lie algebra with non-abelian hypercomplex structure where the Lie algebra of hyperholomorphic vector fields is larger than the center, even though there exists a compatible balanced HKT metric. This shows that the assumption that the hypercomplex structure is abelian in Theorem \ref{abeliancenter} is necessary. The example is due to Barberis and Fino \cite[Section 3.1]{BF}. It is a $12$-dimensional solvable Lie algebra $\mathfrak{g}=\langle e_1,\dots,e_{12}\rangle$ with structure equations
 $$\left\{
    \begin{array}{l}
        de^1=de^5=de^6=de^7=de^8=0,  \\
        de^2=-e^5\wedge e^6+e^7\wedge e^8,\\
        de^3=-e^6\wedge e^8-e^5\wedge e^7,\\
        de^4=e^6\wedge e^7-e^5\wedge e^8,\\
        de^9=e^1 \wedge e^{10},\\
        de^{10}=-e^1\wedge e^9,\\
        de^{11}=e^1\wedge e^{12},\\
        de^{12}=-e^1\wedge e^{11},
    \end{array}
\right.$$
and hypercomplex structure
\[
Ie^1=e^2, \quad Ie^3=e^4,\quad Ie^5=e^6, \quad Ie^7=e^8, \quad Ie^9=e^{10}, \quad Ie^{11}=e^{12},
\]
\[
Je^1=e^3, \quad Je^2=-e^4,\quad Je^5=e^7, \quad Je^6=-e^8, \quad Je^9=e^{12}, \quad Je^{10}=-e^{12}.
\]
The center of left-invariant $ (1,0) $-vector fields is generated by $Z=e_3-\sqrt{-1}e_4$ whereas the space of hyperholomorphic vector fields also contains $J\bar Z $.
\end{ex}

\section{HKT--Einstein metrics and hyperholomorphic vector fields}\label{Sec:Einstein}
In this section we show the non-existence of hyperholomorphic $(1,0)$-vector fields on compact HKT--Einstein manifolds with non-zero Einstein constant. Such a result follows from a sort of Bochner-type formula proved in Proposition \ref{Prop:Bochner}. We start with a preliminary lemma:

\begin{lem}\label{delta}
Let $(M,I,J,K,g,\Omega)$ be a HKT manifold and $\nabla$ the Obata connection on it. Then for every vector field $X\in TM$ we have
\begin{equation}\label{eq:delta_1}
\delta X^\flat=-\mathrm{tr}(\nabla X)+2\theta(X)\,.
\end{equation}
\end{lem}
\begin{proof}
Let $\nabla^{\mathrm{LC}}$ and $D^b$ be the Levi-Civita and the Bismut connections of $g$, respectively. Then
\[
g(D^b_XY,Z)=g(\nabla^{\mathrm{LC}}_XY,Z)+\frac{1}{2}T(X,Y,Z),
\]
where $T$ is the torsion of the Bismut connection. We also recall the formula in \cite[Proposition 3.1]{IP}:
\[
g(\nabla_XY,Z)=g(D^b_XY,Z)+A(X,Y,Z),
\]
where
\[
A(X,Y,Z):=-\frac{1}{2}\left( T(X,IY,IZ)+T(IX,IY,Z)+T(X,KY,KZ)+T(IX,KY,JZ) \right).
\]
Choose an orthonormal basis $e_1,\dots,e_{4n}$ of $TM$. Then
\[
\mathrm{tr}(\nabla X)=\sum_{i=1}^{4n}g(\nabla_{e_i}X,e_i)=\mathrm{tr}(\nabla^{\mathrm{LC}} X)+\sum_{i=1}^{4n}A(e_i,X,e_i).
\]
Since $\mathrm{tr}(\nabla^{\mathrm{LC}} X)=\sum_{i=1}^{4n} (\nabla^{\mathrm{LC}}_{e_i}X)(e_i)=-\delta X^\flat$, we only need to show that $\sum_{i=1}^{4n}A(e_i,X,e_i)=2\theta(X)$. Using the fact that $T$ is a $3$-form we deduce
\[
T(e_i,IX,Ie_i)+T(Ie_i,IX,e_i)=0, \qquad \text{for all }i=1,\dots,4n
\]
and
\[
\sum_{i=1}^{4n}\left( T(e_i,KX,Ke_i)+T(Ie_i,KX,Je_i) \right)=-2\sum_{i=1}^{4n}T(KX,e_i,Ke_i)=-4\theta(X),
\]
where the last equality can be found, e.g., in \cite{IPap}.
\end{proof}

\begin{prop}\label{Prop:Bochner}
Let $(M,I,J,K,g,\Omega)$ be a HKT manifold. Suppose that $Z\in T^{1,0}_I(M)$ is hyperholomorphic and denote by $\alpha\in \Lambda^{1,0}_I(M)$ the form given by $\alpha=J\left(Z\right)^{\flat}$. Then
\begin{equation}\label{Boc}
\mathcal{L}_Z(\partial^* J \bar \alpha)-\mathcal{L}_{J\bar Z}(\partial^* \alpha)= \left( \rho-J\rho\right)(Z,I\bar Z).
\end{equation}
If furthermore $M$ is compact
\begin{equation}\label{Bochner}
\| \partial^* \alpha \|^2_{L^2}+ \| \partial^* J\bar \alpha \|^2_{L^2}+\int_M \left( \rho-J\rho \right)(Z,I\bar Z) \mathrm{Vol}_g=0.
\end{equation}
\end{prop}
\begin{proof}
Thanks to Lemma \ref{delta} and Proposition \ref{parallel-obata}, we have
\begin{eqnarray*}
\mathcal{L}_Z(\partial^* \bar Z^\flat)+\mathcal{L}_{J\bar Z}(\partial^* (JZ^\flat))&=&2\mathcal{L}_Z(\theta^{0,1}(\bar Z))+2\mathcal{L}_{J\bar Z}(\theta^{0,1}(JZ)),\\
&=&2\mathcal{L}_Z(\theta^{0,1})(\bar Z)+2\mathcal{L}_{J\bar Z}(\theta^{0,1})(JZ),\\
&=&2 \partial (\theta^{0,1})(Z,\bar Z)+2\partial (\theta^{0,1})(J\bar Z,JZ),\\
&=&2 \sqrt{-1} \left( \partial (\theta^{0,1})-J\partial (\theta^{0,1})\right)(Z,I\bar Z),\\
&=&- \left( \rho-J\rho\right)(Z,I\bar Z),
\end{eqnarray*}
where the last identity is as in the proof of \cite[Lemma 4.2]{FG}. In the compact case, integrating this identity we conclude
\begin{eqnarray*}
-\int_M \left( \rho-J\rho \right)(Z,I\bar Z) \mathrm{Vol}_g &=&
\int_M \left( \mathcal{L}_Z(\partial^* \bar Z^\flat)+\mathcal{L}_{J\bar Z}(\partial^* (JZ^\flat)) \right) \mathrm{Vol}_g,\\
&=&\int_M \left( g(\partial \partial^* \bar Z^\flat,\bar Z^\flat)+g(\partial \partial^* (JZ^\flat), JZ^\flat) \right) \mathrm{Vol}_g,\\
&=& \int_M \left( (\partial^* \bar Z^\flat)^2+ (\partial^* (JZ^\flat) )^2 \right) \mathrm{Vol}_g,\\
&=& \| \partial^* J\bar \alpha \|^2_{L^2}+\| \partial^* \alpha \|^2_{L^2},
\end{eqnarray*}
and formula \eqref{Bochner} is proved.
\end{proof}

From Proposition \ref{Prop:Bochner} we recover that in the balanced HKT case $\alpha$ and $J\bar \alpha$ are $\partial^*$-closed furthermore, for HKT--Einstein metrics that are not balanced we infer:

\begin{thm}\label{Thm:Einst}
Let $(M,I,J,K,g,\Omega)$ be a compact HKT--Einstein metric with non-zero Einstein constant $\lambda$. Then there are no non-trivial hyperholomorphic $(1,0)$-vector fields on $(M,I,J,K)$.
\end{thm}
\begin{proof}
From Proposition \ref{Prop:Bochner} we deduce
\[
0\geq \left( \rho-J\rho \right)(Z,I\bar Z)=\lambda \omega (Z,I\bar Z)=\lambda \|Z\|^2
\]
and since in the compact setting necessarily $\lambda>0$ \cite{FG} we conclude $Z=0$.
\end{proof}




\begin{rmk}
The analogue of Theorem \ref{Thm:Einst} is false for real hyperholomorphic vector fields. Indeed, the Hopf surface (Example \ref{Hopf}) admits a real hyperholomorphic vector field and yet it can be equipped with a compatible HKT--Einstein metric with non-zero Einstein constant as shown in \cite[Section 8.1]{FG}.
\end{rmk}


\begin{thebibliography}{[99]}
\bibitem{AV}
{\sc Alesker, S.; Verbitsky, M.}, Quaternionic Monge-Ampère equation and Calabi problem for HKT-manifolds, {\em Israel J. Math.} {\bf 176} (2010), 109--138.

\bibitem{AB2}
{\sc Andrada, A.; Barberis, M. L.}, Applications of the quaternionic Jordan form to hypercomplex geometry, {\em J. Algebra} {\bf 664} (2025), 73--122.

\bibitem{AB1}
{\sc Andrada, A.; Barberis, M. L.}, Hypercomplex almost abelian solvmanifolds, {\em J. Geom. Anal.} {\bf 33} (2023), no. 7, Paper No. 213, 31 pp.


\bibitem{AGT}
{\sc Andrada A.; Garrone A.; Tolcachier A.}, Hypercomplex structures on special linear groups, e-print 2024, \href{https://arxiv.org/abs/2408.14715}{\tt arXiv:2408.14715}.

\bibitem{AT}
{\sc Andrada A.; Tolcachier A.}, On the canonical bundle of complex solvmanifolds and applications to hypercomplex geometry, to appear in {\em Transform. groups} (2024).


\bibitem{BDV}
{\sc Barberis, M. L.; Dotti, I. G.; Verbitsky, M.}, Canonical bundles of complex nilmanifolds, with applications to hypercomplex geometry, {\em Math. Res. Lett.} {\bf 16} (2009), no. 2, 331--347.

\bibitem{BF}
{\sc Barberis, M. L.; Fino, A.}, New HKT manifolds arising from quaternionic representations, {\em Math. Z.} {\bf 267} (2011), no. 3--4, 717--735.

\bibitem{BGV}
{\sc Bedulli, L.; Gentili, G.; Vezzoni, L.}, The parabolic quaternionic Calabi-Yau equation on hyperk\"ahler manifolds, {\em Rev. Mat. Iberoam.} {\bf 40} (2024), no. 6, 2291--2310.

\bibitem{BFG}
{\sc Brienza, B.; Fino, A.; Grantcharov, G.}, A mapping tori construction of strong HKT and generalized hyperk\"ahler manifolds, 2025, to appear in the volume {\em Real and Complex Geometry - in Honour of Paul Gauduchon}. 

\bibitem{BFGV}
{\sc Brienza, B.; Fino, A.; Grantcharov, G.; Verbitsky, M.}, On the structure of compact strong HKT manifolds, e-print 2025, \href{https://arxiv.org/abs/2505.06058}{\tt arXiv:2505.06058}.

\bibitem{DS}
{\sc Dinew, S.; Sroka M.}, On the Alesker-Verbitsky conjecture on hyperK\"ahler manifolds. {\em  Geom. Funct. Anal.} {\bf 33} (2023), no. 4, 875--911.


\bibitem{FG}
{\sc Fusi, E.; Gentili, G.}, Special metrics in hypercomplex geometry, e-print 2024, \href{https://arxiv.org/abs/2401.13056}{\tt arXiv:2401.13056}.

\bibitem{GI}
{\sc Ganchev, G.; Ivanov, S.}, Holomorphic and Killing vector fields on compact balanced Hermitian manifolds. {\em Int. J. Math.} {\bf 11} (2000), no. 1, 15--28. 


\bibitem{G}
{\sc Gauduchon, P.}, La 1-forme de torsion d'une vari\`et\`e hermitienne compacte, {\em Math. Ann.} {\bf 267} (1984), no. 4, 495--518.

\bibitem{GTar}
{\sc Gentili, G.; Tardini, N.}, HKT manifolds: Hodge theory, formality and balanced metrics. {\em Q. J. Math.} {\bf 75} (2024), no. 2, 413--435.


\bibitem{GPP}
{\sc Grantcharov D.; Papadopoulos G.; Poon Y.S.} Reduction of HKT-structures, {\em J. Math. Phys.} {\bf  43} (2002), no. 7, 3766--3782.

\bibitem{GLV}
{\sc Grantcharov, G.; Lejmi, M.; Verbitsky, M.}, Existence of HKT metrics on hypercomplex manifolds of real dimension 8. {\em Adv. Math.} {\bf 320} (2017), 1135--1157.

\bibitem{HKLR}
{\sc Hitchin, N. J.; Karlhede, A.; Lindstr\"om, U.; Ro\v{c}ek, M.}, Hyper-K\"ahler metrics and supersymmetry, {\em Comm. Math. Phys.} {\bf 108} (1987), no. 4, 535--589.

\bibitem{HP}
{\sc Howe, P. S.; Papadopoulos G.}, Twistor spaces for hyper-K\"{a}hler manifolds with torsion, { \em Phys. Lett. B}, {\bf 379} (1996), 80--86.


\bibitem{IPap}
{\sc Ivanov, S.; Papadopoulos, G.}, Vanishing theorems and string backgrounds, {\em Classical Quantum Gravity} {\bf 18} (2001), no. 6, 1089--1110.

\bibitem{IP}
{\sc Ivanov, S.; Petkov, A.}, HKT manifolds with holonomy $\mathrm{SL}(n,\H) $, {\em Int. Math. Res. Not.} {\bf 2012} (2012), no. 16, 3779--3799. 


\bibitem{Joycered}
{\sc Joyce, D.}, The hypercomplex quotient and the quaternionic quotient, {\em Math. Ann.} {\bf 290} (1991), no. 2, 323--340.


\bibitem{LBS}
{\sc LeBrun, C.; Simanca, S. R.}, Extremal K\"ahler metrics and complex deformation theory, {\em Geom. Funct. Anal.} {\bf 4} (1994), no. 3, 298--336.


\bibitem{LT}
{\sc Lejmi, M.; Tardini, N.}, On the invariant and anti-invariant cohomologies of hypercomplex manifolds, 2024, to appear in {\em Transform. Groups.}

\bibitem{LW}
{\sc Lejmi, M.; Weber, P.}, Quaternionic Bott-Chern cohomology and existence of HKT metrics. {\em Q. J. Math.} {\bf 68} (2017), no. 3, 705--728.

\bibitem{Obata}
{\sc Obata, M.}, Affine connections on manifolds with almost complex, quaternionic or Hermitian structures. { \em Japan. J. Math.} {\bf 26} (1956), 43--79.

\bibitem{OPS}
{\sc Ornea, L.; Poon, Y. S.; Swann, A.}, Potential $1$-forms for hyper-K\"ahler structures with torsion. {\em Classical Quantum Gravity} {\bf 20} (2003), no. 9, 1845--1856.

\bibitem{Pap1}
{\sc Papadopoulos, G.}, Scale and Conformal Invariance in Heterotic $\sigma$-Models, e-print 2024, \href{https://arxiv.org/abs/2409.01818}{\tt arXiv:2409.01818}.

\bibitem{Pap2}
{\sc Papadopoulos, G.}, Geometry and symmetries of Hermitian-Einstein and instanton connection moduli spaces,   e-print 2025, \href{https://arxiv.org/abs/2501.09474}{\tt arXiv:2501.09474}.

\bibitem{PPS}
{\sc Pedersen, H.; Poon, Y. S.; Swann, A. F.}, Hypercomplex structures associated to quaternionic manifolds, {\em Differential Geom. Appl.} {\bf  9} (1998), no. 3, 273--292.

\bibitem{Sol}
{\sc Soldatenkov, A.}, Holonomy of the Obata connection on $SU(3)$, {\em Int. Math. Res. Notices} (2012), Vol. 2012 (15), 3483--3497.

\bibitem{Ver02}
{\sc Verbitsky, M.}, HyperK\"ahler manifolds with torsion, supersymmetry and Hodge theory, {\em Asian J. Math. } {\bf 6} (2002), no. 4, 679--712.

\bibitem{Ver03}
{\sc Verbitsky, M.}, Hyperk\"ahler manifolds with torsion obtained from hyperholomorphic bundles, {\em Math. Res. Lett. } {\bf 10} (2003), no. 4, 501--513.

\bibitem{Ver07}
{\sc Verbitsky, M.}, Hypercomplex manifolds with trivial canonical bundle and their holonomy. {  Moscow Seminar on Mathematical Physics}. II, 203--211,	{ \em   Amer. Math. Soc. Transl. Ser. 2}, {\bf 221}, {  \em Adv. Math. Sci.}, {\bf 60}, {\em   Amer. Math. Soc., Providence, RI}, 2007.

\bibitem{Ver09}
{\sc Verbitsky, M.}, Balanced HKT metrics and strong HKT metrics on hypercomplex manifolds, {\em Math. Res. Lett.} {\bf 16} (2009), no. 4, 735--752.


\bibitem{Witten}
{\sc Witten, E.}, Instantons and the Large $N=4$ Algebra,  e-print 2024, \href{https://arxiv.org/abs/2407.20964}{\tt arXiv:2407.20964}.

\end{thebibliography}
\end{document}